\documentclass{article}  
\usepackage{amsmath,amssymb,amsthm,bbm}
\usepackage{mathrsfs}
\usepackage{bbm} 
\usepackage[usenames,dvipsnames]{color}
\usepackage[affil-it]{authblk}
\usepackage{tikz}
\usepackage{tikz-cd}
\title{\textbf{The absolute of finitely generated groups: I.~Commutative (semi)groups}\thanks{Partially supported by the RFBR grant
17-01-00433.}}

\author{A.\,M.~Vershik$^{\alpha,\beta,\gamma,}$\thanks{e-mail: avershik@gmail.com} ~and 
A.\,V.~Malyutin$^{\alpha,\beta,}$\thanks{e-mail: malyutin@pdmi.ras.ru}}

\affil{\small
$^{\alpha}$St.~Petersburg Department of Steklov Institute of Mathematics,\\
$^{\beta}$St.~Petersburg State University,\\
$^{\gamma}$Institute for Information Transmission Problems, Moscow}

\date{}

\newcounter{numsec}[section]

\newtheoremstyle{Mystyle}
     {\topsep}
     {\topsep}
     {\it}
     {}
     {\bfseries}
     { }
     { }
     {\thmnumber{#2.~}\thmname{#1}\thmnote{ #3}.}
\theoremstyle{Mystyle}
\newtheorem{theorem}[numsec]{Theorem}
\newtheorem{proposition}[numsec]{Proposition}
\newtheorem{corollary}[numsec]{Corollary}
\newtheorem{claim}[numsec]{Claim}

\newcommand \be     {\begin{equation}}
\newcommand \ee     {\end{equation}}

\newcommand \N      {\mathbb N}
\newcommand \Z      {\mathbb Z}

\newcommand \R      {\mathbb R}

\newcommand \dd     {\partial}
\newcommand \inr    {\operatorname{int}}

\renewcommand \ge{\geqslant}

\newcommand   \rank     {\operatorname{rank}}

\newcommand \card     {\operatorname{card}}

\newcommand \Paths   {\operatorname{\mathscr{P}}}

\newcommand \Rc      {\mathcal{R}^c}
\newcommand \bRr      {\bar{\mathcal{R}}}
\newcommand \bRc      {\bar{\mathcal{R}}^c}

\newcommand \VV      {V_0}

\newcommand \Damma   {\operatorname{D}}

\newcommand \sch   {\sigma_{\scriptscriptstyle S}(G)}
\newcommand \bsch   {\bar{\sigma}_{\scriptscriptstyle S}(G)}

\begin{document}

\maketitle

\begin{abstract}
We give a complete description of the absolute of commutative finitely generated groups and semigroups. The absolute (previously called the exit boundary) is a further elaboration of the notion of the boundary of a random walk on a group (the Poisson--Furstenberg boundary); namely, the absolute of a (semi)group is the set of ergodic central measures on the compactum of all infinite trajectories of a simple random walk on the group. Related notions have been discussed in the probability literature: Martin boundary, entrance and exit boundaries (Dynkin), central measures on path spaces of graphs (Vershik--Kerov). A central measure (with respect to a finite system of generators of a group or semigroup) is a Markov measure on the space of trajectories whose cotransition distribution at every point is the uniform distribution on the generators (i.\,e., a measure of maximal entropy). For a more general notion of measures with a given cocycle, see~\cite{Â15}. For the group~$\Bbb Z$, the problem of describing the absolute is solved exactly by the classical de~Finetti's theorem. The main result of this paper, which is a far-reaching generalization of de~Finetti's theorem, is as follows: the absolute of a commutative semigroup coincides with the set of central measures corresponding to (nonstationary) Markov chains with independent identically distributed increments. Topologically, the absolute is (in the main case) a closed disk of finite dimension.
\end{abstract}

\section{Introduction}

The problem of describing the set of all Borel measures satisfying some invariance condition is typical for several areas of mathematics (probability, dynamical systems, graph theory, representation theory, etc.). Its most general setting presumes the existence of some equivalence relation on a Borel space (for instance, the orbit partition for a group action) and a 2-cocycle on this equivalence relation, and the problem is to describe all probability measures for which this cocycle is the Radon--Nikodym cocycle.

If the equivalence classes are countable, the equivalence relation itself is hyperfinite (i.\,e., is a monotone limit of finite equivalence relations), and the cocycle is identically~$1$, then the problem reduces to describing the so-called central measures (see below) on the space of infinite paths of a graded graph  (Bratteli diagram). In a certain sense, the notion of centrality coincides with the notion of invariance, namely, if we introduce the transformation of paths called the adic shift, then the centrality of a measure coincides with its invariance under this shift. It follows, in particular, that the set of central measures is a Choquet simplex.

\emph{The set of all ergodic central measures for a given equivalence relation, endowed with the weak topology of the space of all Borel measures, is called the absolute} (for more details on the setting and the history of the problem, see~\cite{V14a,Â14b}).

The special case of this problem  considered in this paper is that of finding the   \emph{absolute} for random walks on groups, semigroups, and for dynamic graphs. We develop an approach and present a solution of the problem for an important special case, namely, for random walks on countable commutative groups and semigroups. This case has important special properties as compared with the general case; the details are discussed below.

Note that the definition of absolute resembles the definition of boundaries in potential theory or the theory of random walks (the Poisson--Furstenberg (PF) boundary, Martin boundary, etc.; see, e.\,g.,~\cite{KV83}). And indeed, the absolute is a generalization, or, better to say, a refinement of the PF boundary; more exactly, it is nontrivial if and only if the PF boundary is trivial, i.\,e., consists of a single point; hence the absolute provides additional information on random walks on groups.

The foundations of the theory of absolute were laid in~\cite{V14a,Â14b,Â15}.
In~\cite{ÂÌ15}, a description of the absolute for the case of free groups and homogeneous trees is obtained. This paper deals with the opposite class of groups, that of commutative groups and semigroups. In a paper in preparation, we will consider the next case: the absolute of nilpotent groups and, in particular, Heisenberg groups; this case seems to be much more complicated and interesting.

Let us discuss in more detail  what is meant by a description of the absolute. The absolute is defined as a collection of measures. Measures on the compactum of infinite paths admit a direct description in terms of their values on finite paths (i.\,e., on cylinder sets corresponding to finite paths). For central measures, there is a more concise description in terms of functions on the set of vertices of the dynamic graph. With this approach, the absolute corresponds to classes of proportional minimal nonnegative harmonic functions on the dynamic graph. Another form of describing the absolute, in terms of the transition probabilities of  a Markov chain, appears since the random process corresponding to a central measure is Markovian. It is this description that is most convenient for our constructions.

The key result of this paper is Theorem~\ref{th:Levyprocess-AM}, which says that in the case of a commutative semigroup (with an arbitrary system of generators), the set of ergodic central measures  {\rm(}i.\,e., the absolute{\rm)} coincides with the set of central measures that give rise to Markov chains with independent identically distributed (i.i.d.) increments. The transition probabilities of a chain with i.i.d.\ increments are the same for all vertices of the graph, they depend only on the generators assigned to edges.

Theorem~\ref{th:Levyprocess-AM}, on the one hand, generalizes de Finetti's theorem and, on the other hand, is related to known results on harmonic functions on commutative groups (see~\cite{ChD60,DSW60}, and also~\cite[pp.~311--312]{Wo00}, and references therein). De Finetti's theorem follows from Theorem~\ref{th:Levyprocess-AM} if we consider the case of a free semigroup. Harmonic functions are related to the absolute as follows: in the case of a group, there is a natural bijection between the main part of the absolute (for the definition, see Section~\ref{sec:definitions}) and the space of classes of proportional minimal positive eigenfunctions of the Laplace operator.\footnote{Hereafter, given a semigroup~$G$ with a fixed finite system of generators~$S$, by the Laplace operator (Laplacian) we mean the operator on the space of functions on~$G$ that sends a function~$f$ to the function~$L_f$ defined by the formula
$$
L_f(g):=\frac{1}{|S|}\sum_{s\in S}f(gs).
$$}
(We will discuss this in more detail elsewhere.)

The paper is organized as follows. Section~\ref{sec:definitions} contains the basic definitions. Theorem~\ref{th:Levyprocess-AM} is proved in Section~\ref{sec:absolute-abel}, where we also present Theorem~\ref{pr:Levyprocess-equations}, which gives equations for describing the absolute. Theorems~\ref{th:Levyprocess-AM} and~\ref{pr:Levyprocess-equations} allow us to describe the absolute of a commutative semigroup given a set of defining relations. The absolute is described as the set of solutions of a system of equations in a Euclidean space. In the same section, we give a series of examples of such a description. In Section~\ref{sec:absolute-topol}, we use Theorems~\ref{th:Levyprocess-AM} and~\ref{pr:Levyprocess-equations} to derive theorems on the topological structure of the absolute of commutative groups and semigroups. In this case, the absolute is compact; moreover, for groups and cancellative semigroups,  it is a closed disk of finite dimension. The main technical difficulty in the proof of these theorems is to describe the degenerate part of the absolute. A technical result solving this difficulty is placed in a separate Section~\ref{sec:affine-space}.
In Section~\ref{sec:char}, we discuss the relation of the absolute of a commutative group to multiplicative semigroup characters. 

The authors are grateful to V.~Kaimanovich and S.~Podkorytov for valuable comments and discussions.

\section{Necessary definitions}
\label{sec:definitions}

First, we will give another definition of the absolute of a graph, which does not involve the group-theoretic terminology. By a
 \emph{graph} we mean a locally finite directed graph with a distinguished vertex. Loops and multiple edges are allowed. A
\emph{path} in a graph is a (finite or infinite) sequence of alternating vertices and edges of the form
$$
v_0, e_1, v_1, e_2, \dots, e_n, v_n,
$$
where $e_k$ is an edge leading from the vertex~$v_{k-1}$ to the vertex~$v_k$ 
(both vertices and edges may be repeated). We will consider graphs in which there are infinite paths starting at the distinguished vertex.

Let $\Gamma$ be a graph of the above form. Denote by $\Paths_\Gamma$ the set of all infinite paths in~$\Gamma$ starting at the distinguished vertex. This set is compact in the weak topology. We consider Borel probability measures on this space. Given such a measure, by the measure of a finite path~$R$ (starting at the distinguished vertex) we mean the measure of the cylinder of all infinite paths that begin with~$R$. A measure~$\nu$ on~$\Paths_\Gamma$ is called \emph{central} if it has the following property: for every vertex~$v$ of~$\Gamma$ and every positive integer~$n$, the measure~$\nu$ takes the same value at all paths of length~$n$ that lead from the distinguished vertex to~$v$. The set of central measures is a convex compactum, which is a simplex (see~\cite{Â15}) in the compactum of all measures on~$\Paths_\Gamma$.
A central measure is called \emph{ergodic} (or \emph{regular}) if it is an extreme point of this simplex.

The \emph{absolute} of a graph is the set of all ergodic central measures on the compactum of infinite paths starting at the distinguished vertex. The absolute of a finitely generated semigroup with a fixed finite system of generators is the absolute of the corresponding Cayley graph. In this definition, there is a subtlety related to noncancellative semigroups. Recall that a semigroup~$G$ is called  \emph{cancellative} if there are no elements $a,b,c$ in~$G$ such that  $a\neq b$, but $ac=bc$ and/or $ca=cb$.
In the case of the Cayley graph of a cancellative semigroup, the choice of the distinguished vertex does not affect the absolute; but in the case of a noncancellative semigroup~$G$, we will assume that in~$G$ there is an identity element (or it has been added), and it is this element that is chosen as the distinguished vertex. The absolute of a semigroup~$G$ with a system of generators~$S$ is denoted by~$\mathscr{A}_S(G)$.

A measure $\nu$ on the compactum of paths is called \emph{nondegenerate} if the probability of every path is nonzero. The
\emph{main part} of the absolute is its subset  consisting of the nondegenerate measures.
The set of degenerate ergodic central measures will be called the \emph{degenerate part} of the absolute.

A \emph{branching graph} is a graph in which  the set of paths leading from the distinguished vertex to every vertex~$v$ is nonempty (in this case, one says that $v$ is  \emph{reachable} from the distinguished vertex) and all these paths have the same length. On the set of vertices of a branching graph there is a natural grading by the distance to the distinguished vertex. It turns out that in the theory of absolute, graphs of this special form are most general in the following sense. To a graph~~$\Gamma$ with distinguished vertex~$v_0$ we canonically associate the corresponding
\emph{dynamic graph} $\Damma_{v_0}(\Gamma)$, which is a branching graph constructed in the following way. The
$n$th level of $\Damma_{v_0}(\Gamma)$ is a copy of the set of vertices of~$\Gamma$ connected with~$v_0$ by paths of length~$n$. There are exactly $k$ edges leading from a vertex~$v_1$ to a vertex~$v_2$ in $\Damma_{v_0}(\Gamma)$ if and only if the level of~$v_2$ is greater by one than the level of~$v_1$, and there are exactly $k$ edges leading from the vertex~$u_1$ of~$\Gamma$ corresponding to~$v_1$ to the vertex~$u_2$ of~$\Gamma$ corresponding to~$v_2$. Every branching graph is isomorphic to its dynamic graph. The spaces of paths starting at the distinguished point coincide for a graph and its dynamic graph. The absolute of a graph coincides with the absolute of its dynamic graph.

The construction of a branching graph has a counterpart at the level of groups and semigroups. A
\emph{branching monoid} is a monoid (semigroup with identity element) whose Cayley graph (with respect to some system of generators) is a branching graph. A semigroup is called a \emph{branching semigroup} if it is a branching monoid or can be obtained from a branching monoid by removing the identity element. The following property is characteristic for branching semigroups: if $G$ is a branching semigroup with respect to a system of generators~$S$, then for every element of~$G$, all words in the alphabet~$S$ representing this element have the same length (in other words, relations in this case identify only words of equal length). In a branching semigroup there is a canonical set of generators, which consists exactly of all irreducible elements of the semigroup. Systems of generators build from this canonical set (a system of generators may contain repeated elements, i.\,e., include elements  with multiplicities) will be called \emph{admissible}. The Cayley graph of a branching monoid is a branching graph only for an admissible system of generators. To a semigroup~$G$ with a fixed system of generators~$S$ we canonically associate a branching monoid $\Damma_S(G)$ defined as follows: the system of generators of $\Damma_S(G)$ is
a copy of~$S$, and the set of relations is the subset of the full set of relations for~$(G,S)$ consisting of the relations that identify words of equal length.

\section{The absolute of commutative groups and semigroups}
\label{sec:absolute-abel}

As one can easily see, the random process corresponding to a central measure (on an arbitrary graph of the form described above) is Markovian. For Markov chains on the Cayley graph of a semigroup, we introduce the notion of independent identically distributed increments: a Markov chain is said to have \emph{independent identically distributed increments} if its transition probabilities at all edges marked by the same generator are equal.\footnote{In the paper~\cite{Â17} in preparation, a more general notion of \emph{transfer} is introduced; it is a transformation on the path space of a graph with the meaning of a shift of increments. The notion of a Markov chain with independent identically distributed increments can be rephrased as that of a Bernoulli transfer.} For commutative semigroups, the following key theorem holds, which is a far-reaching generalization of de Finetti's theorem.

\begin{theorem}
\label{th:Levyprocess-AM}
For every commutative semigroup with an arbitrary finite system of generators, the set of ergodic central measures {\rm(}i.\,e., the absolute{\rm)} coincides with the set of central measures that give rise to Markov chains with independent identically distributed increments.

Thus the absolute is in a bijective correspondence with the set of measures on the set of generators that determine Markov chains with the above centrality property. An explicit condition that distinguishes these measures is given in Theorem~{\rm\ref{pr:Levyprocess-equations}}.
\end{theorem}

\begin{proof}
For brevity, measures that give rise to Markov chains with independent identically distributed increments will be called
\emph{measures with i.i.d.\ increments}. Let $P$ be a finite path (starting at the distinguished point) in the Cayley graph of the given semigroup (with respect to the given system of generators). The left translation in the semigroup determines a homeomorphism between the subcompactum~$\Paths_P$ of infinite paths that begin with~$P$ and the compactum of all infinite paths (starting at the distinguished vertex), as well as an isomorphism~$\phi_*$ between the spaces of measures on these compacta.\footnote{Examples of noncancellative semigroups  are interesting, but they are partly beyond the context we are interested in. In particular, the Cayley graph of such a semigroup is inhomogeneous:  a pair of edges marked by $s_1$ and $s_2$ may have the same initial vertices and the same final vertices at one segment of the graph, and the same initial vertices but different final vertices at another segment of the graph. In the context of this proof, it is worth mentioning that in the case of a noncancellative semigroup, the tail filtrations on the subcompactum~$\Paths_P$ and on the compactum of all paths are not necessarily isomorphic.} If $\nu$ is an ergodic central measure and $\nu(\Paths_P)>0$,
denote by $\nu_P$ the corresponding conditional measure on~$\Paths_P$.
Then the measure~$\phi_*(\nu_P)$ is also central (since $a=b$ implies $ca=cb$).
As one can easily see, since the semigroup is commutative,  the central measure~$\nu$ dominates the (finite central) measure~$\nu(\Paths_P)\cdot\phi_*(\nu_P)$. By ergodicity, it follows that $\nu=\phi_*(\nu_P)$.
This proves that~$\nu$ is a measure with i.i.d.\ increments. On the other hand, if all ergodic central measures have i.i.d.\ increments, then every central measure with i.i.d.\ increments is ergodic, since noncoinciding measures with i.i.d.\ increments are mutually singular (the problem reduces to the mutual singularity of noncoinciding Bernoulli measures).
\end{proof}

\subsection{Explicit computation of the absolute}
Theorem~\ref{th:Levyprocess-AM} provides a recipe for describing the absolute, and below we carry out its computation.

For an arbitrary semigroup~$G$ with a fixed finite system of generators~$S$, the set~$\mathscr{I}_S(G)$ of measures with i.i.d.\ increments can be identified in a natural way with the simplex~$\Delta_S$ of probability distributions on~$S$:
to a distribution~$\mu$ on $S$ we associate the measure in~$\mathscr{I}_S(G)$ for which the probability of the increment by~$s\in S$ is equal to~$\mu(s)$. A distribution on~$S$ for which the corresponding measure in~$\mathscr{I}_S(G)$ is central will be called
\emph{precentral}. By Theorem~\ref{th:Levyprocess-AM}, in the case of a commutative semigroup, the absolute coincides with the intersection of the  ($S-1$)-dimensional simplex~$\mathscr{I}_S(G)\cong \Delta_S$  (in general, it is not convex in the space of measures on the compactum of paths) with the infinite-dimensional simplex~$\Sigma_S(G)$ of central measures:
$$
\mathscr{A}_S(G)=\Sigma_S(G)\cap \mathscr{I}_S(G).
$$
Thus Theorem~\ref{th:Levyprocess-AM} reduces the problem of describing the absolute~$\mathscr{A}_S(G)$ of a commutative semigroup to the problem of describing the set~$\sch$ of precentral distributions in~$\Delta_S$. The (pre)centrality condition splits into a set of necessary conditions  for pairs of finite paths of equal length leading to the same vertex of the Cayley graph. The following proposition immediately follows from the definition of centrality.

\begin{theorem}
\label{pr:Levyprocess-equations}
A probability distribution~$\mu=\{\mu(s);s\in S\}$ on a finite system of generators~$S$ of a commutative semigroup~$G$ 
is precentral if and only if for every pair of vectors $(m_s)_{s\in S}$ and $(n_s)_{s\in S}$ from $\N_0^S$ such that
\be
\label{eq:central-22}
\sum_{s\in S} m_s=\sum_{s\in S} n_s \quad\text{in~$\N_0$}
\qquad\text{and}\qquad
\sum_{s\in S} m_s\cdot s=\sum_{s\in S} n_s \cdot s \quad\text{in~$G$},
\ee
the following equation holds:
\be
\label{eq:central-pr-2-compact}
\prod_{s\in S}(\mu(s))^{m_s}=\prod_{s\in S}(\mu(s))^{n_s}.
\ee
In other words, the set~$\sch$ of precentral distributions corresponding to the absolute coincides with the set of distributions on~$S$ that are solutions of equations~\eqref{eq:central-pr-2-compact} for all coefficients satisfying~\eqref{eq:central-22}.
\end{theorem}

Theorem~\ref{pr:Levyprocess-equations} allows one to obtain a description of the absolute of a commutative semigroup from the set of defining relations. Conditions~\eqref{eq:central-pr-2-compact} are called the \emph{centrality equations} and are of the main interest in the study of the topology of the absolute. When describing the absolute in the context of Theorem~\ref{pr:Levyprocess-equations}, it is convenient to take into account the following considerations.

1. The pairs of vectors satisfying condition~\eqref{eq:central-22} for given~$G$ and~$S$ form a semigroup (we denote it by~$\Rc_S(G)$; this semigroup describes the relations in the branching monoid $\Damma_S(G)$). To verify the precentrality, it suffices to verify condition~\eqref{eq:central-pr-2-compact} for an arbitrary set of generators of the semigroup~$\Rc_S(G)$. 

2. The pairs of vectors from~$\N_0^S$ consisting of two equal\footnote{The same vector in~$\N_0^S$ can represent different paths of equal length leading to the same vertex.} vectors form a subsemigroup in~$\Rc_S(G)$ (we denote it by~$R_0$). Equations~\eqref{eq:central-pr-2-compact} corresponding to elements from~$R_0$ are trivial, i.\,e., they are identities, so in order to verify the precentrality, it suffices to take an arbitrary set of vectors from~$\Rc_S(G)$ that yields a generating set being combined with~$R_0$. In other words, in order to describe the absolute, it suffices to take the system of equations~\eqref{eq:central-pr-2-compact} for a set of noncommutative relations that is defining for the branching monoid~$\Damma_S(G)$ modulo the commutativity relations.

\paragraph{Examples.}
1. The absolute of the commutative semigroup freely generated by a set of generators~$S$ is represented by the simplex~$\Delta_S$,
since vectors $(m_s)_{s\in S}$ and $(n_s)_{s\in S}$ from $\N_0^S$ satisfy condition~\eqref{eq:central-22} only if they coincide, and in this case equation~\eqref{eq:central-pr-2-compact} becomes an identity. At the level of branching monoids, this fact manifests itself as the absence of noncommutative relations.

2. Let $G=\Z$ and $S=\{+1,-1\}$. In this case, as in Example~1, vectors
 $(m_s)_{s\in S}$ and $(n_s)_{s\in S}$ satisfy condition~\eqref{eq:central-22} only if they coincide, so the absolute~$\mathscr{A}_{\{+1,-1\}}(\Z)$ is homeomorphic to the one-dimensional simplex. At the level of branching monoids, the explanation is that the branching monoid for $(\Z,\{+1,-1\})$ is the free monoid with two generators.

3. Let $G=\Z^2$ and $S=\{(+1,0),(-1,0),(0,+1),(0,-1)\}$.
The semigroup~$\Rc_S(G)$ is generated by the subsemigroup~$R_0$ and the pair $((1,1,0,0),(0,0,1,1))$.
This pair gives rise to the equation  $x_1x_2=x_3x_4$ in~$\R^4$. The absolute is represented by the intersection of the set of solutions of this equation with the simplex
$$\left\{x\in\R^4:\sum x_i=1, x_i\ge0\right\}.$$ 
It is homeomorphic to the closed disk of dimension~$2$ (see Theorem~\ref{th:abel-top-group}).

4. Let $G=\Z^d$ and $S$ be the standard symmetric system of generators. This is a generalization of the previous example. In this case, the system of equations is as follows:
$$
x_1x_2=x_3x_4=\dots=x_{2d-1}x_{2d}.
$$ 

5. For $G=\Z$ and $S=\{0,+6,-1\}$, the centrality relation takes the form $x_1^7=x_2x_3^6$.

6. Let $G$ be the commutative semigroup with three generators $a$, $b$, $c$ and the additional noncommutative relation
 $a+b=a+c$. Then $G$ is a branching semigroup, so the set~$\sch$ is determined by the equation $\mu(a)\mu(b)=\mu(a)\mu(c)$ corresponding to the relation $a+b=a+c$. The absolute is homeomorphic to the tripod~$T$. 

7. If $G$ is the commutative semigroup with generators $a$, $b$, $c$ and the additional noncommutative relations
$a+b=2c$ and $a+c=2b$, then the absolute is disconnected, it consists of two points.

\medskip

\begin{proposition}
\label{pr:abel-torsion}
The absolute of the quotient semigroup of a commutative semigroup by a finite subgroup coincides with the absolute of the semigroup. 
\end{proposition}

\begin{proof}
Assume that there is a semigroup epimorphism $G_1\to G_2$ and we consider the system of generators~$S$ in~$G_2$ inherited from~$G_1$. If a pair of vectors $(m_s)_{s\in S}$ and $(n_s)_{s\in S}$ from $\N_0^S$ satisfies condition~\eqref{eq:central-22} for~$G_1$,
then it also satisfies condition~\eqref{eq:central-22} for~$G_2$, since we deal with a homomorphism. 
Conversely, if an epimorhpism $G_1\to G_2$ corresponds to taking the quotient by a finite subgroup and a pair of vectors $(m_s)_{s\in S}$ and $(n_s)_{s\in S}$ satisfies condition~\eqref{eq:central-22} for~$G_2$, then there is a positive integer~$k$ such that the pair
of vectors $(k\cdot m_s)_{s\in S}$ and $(k\cdot n_s)_{s\in S}$ satisfies condition~\eqref{eq:central-22} for~$G_1$. In view of Theorem~\ref{pr:Levyprocess-equations}, it follows that the sets~$\sigma_{\scriptscriptstyle S}(G_1)$ and~$\sigma_{\scriptscriptstyle S}(G_2)$ of precentral distributions are described by equivalent systems of equations, and Theorem~~\ref{th:Levyprocess-AM} implies the desired assertion.
\end{proof}

\paragraph{Comments.}
1. The idea of shifting used in the proof of Theorem~\ref{th:Levyprocess-AM} appeared in the probability literature of the 1960s in connection with harmonic functions; see~\cite[Theorem~5]{DSW60}, \cite[Lemma~1]{Ìî67}, and also \cite[Lemma~25.2]{Wo00}.

2. Already for nilpotent groups, Theorem~\ref{th:Levyprocess-AM} does not hold for an arbitrary ergodic central measure.

\section{The topology of the absolute of commutative (semi)groups}
\label{sec:absolute-topol}

As we have already mentioned, the set of all central measures is a simplex, and the absolute is its Choquet boundary. A meaningful question is what topology is induced on the absolute by the weak topology in the space of measures on the compactum of paths. In this section, we study the topology of the absolute of commutative (semi)groups. From Theorems~\ref{th:Levyprocess-AM} and~\ref{pr:Levyprocess-equations} we derive (the proofs are given below) the following theorems on the topological structure of the absolute. The most important result is as follows: in the case of groups and cancellative semigroups, the main part of the absolute is the interior of a disk, and the degenerate part is the boundary of this disk.

\begin{theorem}[(on the topology of the absolute of commutative groups)]
\label{th:abel-top-group}
The absolute of a finitely generated commutative group with respect to any finite system of semigroup generators is homeomorphic to the closed disk of dimension equal to the rank of the group. The main part of the absolute corresponds to the interior of the disk.
\end{theorem}

Theorem~\ref{th:abel-top-group} can be extended to cancellative semigroups.

\begin{theorem}[(on the topology of the absolute of cancellative commutative semigroups)]
\label{th:abel-top-cancel}
The absolute of a finitely generated cancellative commutative semigroup~$G$ {\rm(}with respect to any finite system of generators~$S${\rm)} is homeomorphic to the closed disk whose dimension either coincides with the rank of the group of fractions\footnote{The \emph{group of fractions} of a semigroup~$G$ is the group with the same generators as~$G$ in which the relations are all corollaries of the relations in~$G$.} of $G$, or is one less than this rank if $G$ is a branching semigroup and $S$ is an admissible system of generators.
 The main part of the absolute corresponds to the interior of the disk.
\end{theorem}

In the case of a noncancellative commutative semigroup, the absolute can have a more complex structure (see Example~4 above). However, it is still compact, and the main part has the same form.

\begin{theorem}[(on the topology of the absolute of  noncancellative commutative semigroups)]
\label{th:abel-top-semi}
The absolute of an arbitrary finitely generated commutative semigroup
{\rm(}with respect to any finite system of generators{\rm)} is compact\footnote{In the literature, a simplex whose Choquet boundary is closed is called a \emph{Bauer} simplex.}, and its main part is homeomorphic to the open disk whose dimension is determined by the rule described in Theorem~{\rm\ref{th:abel-top-cancel}}.
\end{theorem}

\begin{proof}[Proof of Theorems~{\rm\ref{th:abel-top-group}--\ref{th:abel-top-semi}}]
Let $G$ be a commutative semigroup with a finite system of generators~$S$.
Theorems~\ref{th:Levyprocess-AM} and~\ref{pr:Levyprocess-equations} reduce the problem of describing the absolute~$\mathscr{A}_S(G)$ to that of describing the set~$\sch$ of solutions of equations~\eqref{eq:central-pr-2-compact} in the simplex~$\Delta_S$ of probability distributions on~$S$.

{\it A proof that the absolute is compact.} For arbitrary pair of vectors $(m_s)_{s\in S}$ and $(n_s)_{s\in S}$ from $\N_0^S$, the subset of solutions of equation~\eqref{eq:central-pr-2-compact} in~$\Delta_S$ is compact, since $\Delta_S$ is compact and the expressions in both sides of~\eqref{eq:central-pr-2-compact} are continuous on~$\Delta_S$ as functions of~$\mu$. Therefore, the set~$\sch$ (and hence the absolute~$\mathscr{A}_S(G)$) is compact as an intersection of compact subsets.

{\it The centrality equations.}
If vectors $(m_s)_{s\in S}$ and $(n_s)_{s\in S}$ from $\N_0^S$
satisfy the condition $\sum_{s\in S} m_s\cdot s=\sum_{s\in S} n_s \cdot s$ in~$G$,
then the vector $(m_s-n_s)_{s\in S}$ from $\Z^S$ will be called a \emph{reduced relation vector} for~$(G,S)$.
If, in addition, the equation $\sum_{s\in S} m_s=\sum_{s\in S} n_s$ holds, then the reduced relation vector
 $(m_s-n_s)_{s\in S}$ will be called \emph{central},
equation~\eqref{eq:central-pr-2-compact} will be called a \emph{centrality equation}, and the equation (in the variables $\mu(s)$)
\be
\label{eq:central-reduced}
\prod_{s\in S}(\mu(s))^{m_s-\min\{m_s,n_s\}}=\prod_{s\in S}(\mu(s))^{n_s-\min\{m_s,n_s\}}
\ee
will be called a \emph{reduced centrality equation}.\footnote{In the general case of a  noncancellative semigroup, a reduced centrality equation is not necessarily a centrality equation.} Denote by~$\bsch$ the set of distributions from~$\Delta_S$ that are solutions of all reduced centrality equations of the pair~$(G,S)$. Then the following holds.

(i) {\it The set $\bsch$ is contained in $\sch$}, since $\sch$ coincides with the set of solutions of the centrality equations lying in~$\Delta_S$, and every centrality equation~\eqref{eq:central-pr-2-compact} can be obtained from the corresponding reduced equation~\eqref{eq:central-reduced} by multiplying both sides of the latter by $\prod_{s\in S}(\mu(s))^{\min\{m_s,n_s\}}$.

(ii) {\it If $G$ is a  cancellative semigroup, then $\bsch=\sch$}, since in such a semigroup every reduced centrality equation is, obviously, a centrality equation, so $\sch$ is contained in~$\bsch$ (and $\bsch$ is contained in $\sch$ by~(i)).

(iii) {\it The intersection $\sch\cap\inr(\Delta_S)$ coincides with the intersection of the sets $\bsch$ and $\inr(\Delta_S)$, both in the case of a  cancellative semigroup and in the general case}, since in~$\inr(\Delta_S)$ the condition $\mu(s)>0$ holds for all $s\in S$, and a centrality equation in~$\inr(\Delta_S)$ has the same set of solutions as the corresponding reduced centrality equation.

Now observe that the set~$\bRr_S(G)$ of reduced relation vectors for~$(G,S)$ is, obviously, a subgroup in~$\Z^S$,
and the subset~$\bRc_S(G)$ of all central relation vectors either coincides with~$\bRr_S(G)$, or is a subgroup in~$\bRr_S(G)$ of corank~$1$.\footnote{If the group of fractions of a semigroup is torsion-free, then the subgroups~$\bRr_S(G)$ and~$\bRc_S(G)$ are linear subspaces in~$\Z^S$.}

It follows that the sets~$\bRc_S(G)$ and $\bsch$ satisfy the conditions of Proposition~\ref{pr:A7} (in the notation of this proposition, 
$\bsch=\Lambda_{\bRc_S(G)}$), which implies that the set~$\bsch$ is homeomorphic to the closed disk of dimension~$|S|-1-\rank(\bRc_S(G))$ (hereafter, $\rank$ stands for the rank of a commutative group), and, moreover, the interior of the disk~$\bsch$ lies in~$\inr(\Delta_S)$, while the boundary of~$\bsch$ lies  in the boundary $\partial \Delta_S$.

To complete the proof of Theorems~\ref{th:abel-top-group}--\ref{th:abel-top-semi}, it remains to observe that under the bijection $\mathscr{A}_S(G)\cong\sch$, the main part of the absolute is represented by the intersection $\sch\cap\inr(\Delta_S)$,
the rank of the group of fractions of the semigroup~$G$ is equal to $|S|-\rank(\bRr_S(G))$, and the rank of the group $\bRr_S(G)$ coincides with the rank of~$\bRc_S(G)$ if and only if (these groups coincide, so)  $(G,S)$ is a branching semigroup with an admissible system of generators.
\end{proof}

\section{A proposition on linear spaces}
\label{sec:affine-space}

Let $S$ be a finite set,
$\Delta_S$ be the simplex of probability distributions on~$S$,
$\R^S$ be the space of real functions on~$S$,
$\VV\subset \R^S$ be the subset of functions with values summing to zero. We identify
 $\Delta_S$ with the subset of nonnegative functions in~$\R^S$ with values summing to one.
Given a vector~$\kappa=(k_s)_{s\in S}$ from $\VV$, denote by~$\lambda_\kappa$ the subset in~$\Delta_S$ consisting of all distributions~$\mu$ satisfying the condition
\be
\label{def:central-cond-abs}
\prod_{s\in S:~k_s>0}(\mu(s))^{k_s}=\prod_{r\in S:~k_r<0}(\mu(r))^{|k_r|}.
\ee
Given a subset\footnote{If we extend the above definitions to subsets~$K$ not in~$\VV$, then part of the assertions stated below remains valid. However, here we do not consider generalizations; we are mainly interested in the simplex~$\Delta_S$ and the space~$\VV$, as the vector space associated with the affine hull of~$\Delta_S$; these spaces are embedded into the auxiliary space~$\R^S$ only for convenience.} $K$ in~$\VV$, put
$$
\Lambda_K:=\bigcap_{\kappa\in K}\lambda_\kappa.
$$
By $V_K$ we denote the linear hull of~$K$, and $\dim(V)$ stands for the dimension of a space~$V$. We also denote by
 $\inr(M)$ and $\dd M$ the interior and the boundary of a multidimensional polyhedron~$M$ (irrespective of the ambient space).

The purpose of this section is to prove the following proposition.

\begin{proposition}
\label{pr:A7}
If a subset $K$ of a hyperplane~$\VV$ is a linear subspace or a semigroup of vectors with integer coordinates, then the set~$\Lambda_K$ is homeomorphic to the closed disk of dimension~$|S|-1-\dim(V_K)$; moreover, the interior of~$\Lambda_K$ lies in~$\inr(\Delta_S)$, and the boundary of~$\Lambda_K$ lies in~$\partial \Delta_S$.
\end{proposition}

We split the proof of Proposition~\ref{pr:A7} into a series of claims.

\begin{claim}
\label{cl:A2}
For every subset $K$ in~$\VV$, the intersection of the set $\Lambda_K$ with the interior $\inr(\Delta_S)$ of the simplex~$\Delta_S$ is homeomorphic to the open disk of dimension~$|S|-1-\dim(V_K)$.
\end{claim}

\begin{proof}
Observe that the maps
$$\exp\colon \VV\to\inr(\Delta_S)\qquad\text{and}\qquad\log\colon \inr(\Delta_S)\to \VV$$
given by the formulas
$$
\exp((v_s)_{s\in S})=\left(\frac{e^{v_s}}{\sum_{s\in S}e^{v_s}}\right)_{s\in S}~~~~\text{and}
$$
$$
\log((p_s)_{s\in S})=\left(\log p_s-\sum_{s\in S} \log p_s \right)_{s\in S}
$$
are mutually inverse diffeomorphisms between  $\VV$ and~$\inr(\Delta_S)$. Taking the logarithms of both sides of~\eqref{def:central-cond-abs}, we see that 
\emph{the maps $\exp$ and~$\log$ provide a diffeomorphism between the set
$\Lambda_K\cap\operatorname{int}(\Delta_S)$
and the orthogonal complement to~$V_K$ in~$\VV$}. It remains to observe that this complement has dimension
$$\dim (\VV)-\dim(V_K)=|S|-1-\dim(V_K).\qedhere$$
\end{proof}

\begin{claim}
\label{cl:A1}
For every subset~$K$ in~$\VV$, the set~$\Lambda_K$ is compact.
\end{claim}

\begin{proof}
For every vector~$\kappa=(k_s)_{s\in S}$, the set~$\lambda_\kappa$ is compact, since
$\Delta_S$ is compact and the expressions in both sides of~\eqref{def:central-cond-abs} are continuous on~$\Delta_S$ as functions of~$\mu$. Therefore, $\Lambda_K$ is compact as an intersection of compact sets.
\end{proof}

\begin{claim}
\label{cl:A3}
If vectors~$\kappa$ and $\kappa'$ in~$\VV$ are collinear, then $\lambda_\kappa=\lambda_{\kappa'}$.
\end{claim}

\begin{proof}
Follows from the definition of~$\lambda_\kappa$ for $\kappa$ and $\kappa'$ from~$\VV$. \end{proof}

\begin{claim}
\label{cl:A4}
If $K$  is a subgroup in~$\VV$ that contains only vectors with integer coordinates, then $\Lambda_K=\Lambda_{V_K}$.
\end{claim}

\begin{proof}
First, observe that if $K$ is a group of vectors with integer coordinates, then in~$V_K$ every vector with rational coordinates is proportional to a vector from~$K$. In view of Claim~\ref{cl:A3}, it follows that every distribution~$\mu$ from $\Lambda_K$ belongs also to $\lambda_\kappa$ if
$\kappa$ is a vector in~$V_K$ with rational coordinates.

Second, it is clear that for every fixed distribution~$\mu\in\Delta_S$, both sides of~\eqref{def:central-cond-abs},
regarded as functions of a vector $(k_s)_{s\in S}$, are continuous on each connected component of each stratum of the form
$$
\R^S_m:=\{(k_s)_{s\in S}\in\R^S : \card\{s\in S:k_s=0\}=m\}.
$$

Finally, for every $m\in\N_0$, the vectors with rational coordinates are dense in the stratum~$V_K\cap\R^S_m$
(this follows from the fact that the intersection of two linear subspaces spanned by vectors with integer coordinates is also spanned by vectors with integer coordinates; this fact becomes obvious if one regards the intersection as the orthogonal complement to the sum of the orthogonal complements to the original subspaces and observes that the property of being spanned by vectors with integer coordinates is preserved under taking orthogonal complements and sums).
\end{proof}

\begin{claim}
\label{cl:A5}
If an affine line~$L$ in~$\R^S$ intersects the simplex~$\Delta_S$ by an interval~$I$ with endpoints~$a$ and~$b$, then the set~$\lambda_{a-b}$ intersects~$L$ at a single point, and this point lies in the interior of~$I$.
\end{claim}

\begin{proof}
For the vector $\kappa=a-b$, condition~\eqref{def:central-cond-abs} takes the form
\be
\label{eq:Lambda_ab}
\prod_{s\in S:~a_s-b_s>0}(\mu(s))^{a_s-b_s}=\prod_{r\in S:~b_r-a_r>0}(\mu(r))^{b_r-a_r}.
\ee
As $\mu$ moves along the interval~$I$ from~$a$ to~$b$, the function~$F_L(\mu)$ given by the left-hand side of~\eqref{eq:Lambda_ab} decreases, while the function~$F_R(\mu)$ given by the right-hand side of~\eqref{eq:Lambda_ab} increases. Since $I$ is cut out of~$\Delta_S$ by a straight line, there is
$q\in S$ with $a_q=0$ and $b_q>0$ and there is $r\in S$ with $a_r>0$ and $b_r=0$
such that $F_L(b)=0$ and $F_R(a)=0$. This immediately implies the desired assertion.
\end{proof}

\begin{claim}
\label{cl:A6}
If $K$ is a linear subspace in~$\VV$, then taking the quotient $\rho\colon \R^S\to\R^S/K$ yields a homeomorphism between the set $\Lambda_K$ and the convex polyhedron~$\rho(\Delta_S)$; moreover, the intersection~$\Lambda_K\cap\partial\Delta_S$ is homeomorphic to the boundary $\partial(\rho(\Delta_S))$.
\end{claim}

We split the proof of Claim~\ref{cl:A6} into several parts.

\begin{claim}
\label{cl:A6.1}
Under the conditions of Claim~{\rm\ref{cl:A6}}, the restriction of  $\rho$ to~$\Lambda_K$ is injective.
\end{claim}

\begin{proof}
Aiming at a contradiction, assume that in  $\Lambda_K$ there are distinct points
$x$ and $y$ with $\rho(x)=\rho(y)$.
Then the vector $x-y$ lies in~$K$, so $\{x,y\}\subset\Lambda_K\subset\lambda_{x-y}$, which contradicts Claim~\ref{cl:A5} if we choose $L$ to be the affine line passing through the points~$x$ and~$y$ (see also Claim~\ref{cl:A3}).
\end{proof}

\begin{claim}
\label{cl:A6.2}
Under the conditions of Claim~{\rm\ref{cl:A6}}, the map~$\rho$ sends the set $\Lambda_K\cap\partial\Delta_S$ to the boundary $\partial(\rho(\Delta_S))$ of the polyhedron~$\rho(\Delta_S)$.
\end{claim}

\begin{proof}
Aiming at a contradiction, assume that in $\Lambda_K\cap\partial\Delta_S$
there is a point~$b$ with $\rho(b)$ in $\inr(\rho(\Delta_S))$. Observe that the interior of the polyhedron $\rho(\Delta_S)$ is covered by the interior of the polyhedron~$\Delta_S$ (and, moreover, $\rho(\inr(\Delta_S))=\inr(\rho(\Delta_S))$).
Therefore, in~$\inr(\Delta_S)$ there is a point~$y$ with $\rho(y)=\rho(b)$.
Then the vector $b-y$ lies in~$K$.
By Claim~\ref{cl:A5}, the point~$b$ does not lie in~$\lambda_{b-y}$, since it is an endpoint of the interval cut out of~$\Delta_S$ by the straight line passing through~$b$ and~$y$ (see also Claim~\ref{cl:A3}). This contradicts the assumption that $b$ lies in~$\Lambda_K$.
\end{proof}

\begin{claim}
\label{cl:A6.3}
Under the conditions of Claim~{\rm\ref{cl:A6}}, the image $Q:=\rho(\Lambda_K\cap\operatorname{int}(\Delta_S))$ coincides with the interior $\inr(\rho(\Delta_S))$ of the polyhedron~$\rho(\Delta_S)$.
\end{claim}

\begin{proof}
We have proved that the restriction of~$\rho$ to $\Lambda_K$ is injective (Claim~\ref{cl:A6.1}), and the set $\Lambda_K\cap\operatorname{int}(\Delta_S)$ is homeomorphic to the open disk of dimension~$|S|-1-\dim (K)$ (see Claim~\ref{cl:A2}), which coincides with the dimension of the polyhedron~$\rho(\Delta_S)$. As is well known, 
\emph{the image of a continuous embedding of a Euclidean space into itself is open} (Brouwer's invariance of domain theorem). It follows that the set~$Q$ is contained in $\inr(\rho(\Delta_S))$ and open in~$\rho(\Delta_S)$. Therefore, if $Q$ did not cover the domain $\operatorname{int}(\rho(\Delta_S))$, then there would be a point~$x$ in this domain that does not belong to~$Q$ but belongs to the closure of~$Q$. Since $\Lambda_K$ is compact, this would mean that in~$\Lambda_K\cap\partial\Delta_S$ there is a point~$b$ with $\rho(b)=x$. However, this contradicts Claim~\ref{cl:A6.2}.
\end{proof}

\begin{proof}[Completing the proof of Proposition~{\rm\ref{cl:A6}}]
Since $\Lambda_K$ is compact (see Claim~\ref{cl:A1}), it follows from Claim~\ref{cl:A6.3} that $\rho(\Lambda_K)=\rho(\Delta_S)$.
Since the restriction of~$\rho$ to~$\Lambda_K$ is injective (see Claim~\ref{cl:A6.1}), it follows that $\rho$ yields a bijection between~$\Lambda_K$ and~$\rho(\Delta_S)$. It remains to observe that, as one can easily see, a continuous bijection of a metric compact space is a homeomorphism. Thus $\rho$ yields a homeomorhism between $\Lambda_K$ and $\rho(\Delta_S)$, and, in view of Claim~\ref{cl:A6.3}, a homeomoprhism between the spheres~$\Lambda_K\cap\partial\Delta_S$ and $\partial(\rho(\Delta_S))$.
\end{proof}

\begin{proof}[Completing the proof of Proposition~{\rm\ref{pr:A7}}]
Claim~\ref{cl:A4} reduces the situation to the case of a linear subspace. In this case, the existence of a homeomorphism follows from Claim~\ref{cl:A6}, and the refinement on the dimension, from Claim~\ref{cl:A2}.
\end{proof}

\paragraph{Remark.}
If a subset~$K$ of the space~$\VV$ consists of vectors with integer coordinates but is not a subgroup, then the set
$\Lambda_K$ is not necessarily homeomorphic to a disk. For instance, for
 $S=\{1,2,3\}$, we have
$$
\Lambda_{\{(1,1,-2),(1,-2,1)\}}=\{(1/3,1/3,1/3), (1,0,0)\}.
$$

\section{Characters}
\label{sec:char}

The theory of the absolute of commutative groups and semigroups has an interesting reformulation in terms of characters of semigroups.  Here by a \emph{character} we mean a homomorhism to the multiplicative group of nonnegative real numbers. For a general theory of characters, see~\cite{CP61, Ëå70, Ëå71}.

Let $G$ be an arbitrary commutative semigroup with a finite system of generators~$S$
and $\Damma_S(G)$ be the branching monoid for the pair $(G,S)$. Every central measure~$\nu$ of the pair~$(G,S)$ gives rise to the functional $f_\nu$ on the monoid~$\Damma_S(G)$ whose value at an element of~$\Damma_S(G)$ is equal to the measure of a path leading to this element. Clearly, under this correspondence, functions corresponding to central measures with i.i.d.\ increments are exactly characters of the monoid~$\Damma_S(G)$ whose values on the generators from~$S$ sum to 1 (such characters will be called \emph{normalized}, or \emph{probability characters}).
In these terms, Theorem~\ref{th:Levyprocess-AM} takes the following form.

\begin{corollary}
In a commutative semigroup~$G$ with an arbitrary finite system of generators~$S$, the above correspondence
$\nu\mapsto f_\nu$  is a bijection of the set of ergodic central measures
 {\rm(}i.\,e., the absolute{\rm)} onto the set of normalized $\R_{0+}$\nobreakdash-characters on the branching monoid~$\Damma_S(G)$.
\end{corollary}

\begin{proof}
If $\nu$ is ergodic, then the functional~$f_\nu$ is a normalized $\R_{0+}$-character, since the increments are i.i.d.\ 
{\rm(}Theorem~{\rm\ref{th:Levyprocess-AM})}. Conversely, if a functional~$f$ on~$\Damma_S(G)$ is a normalized $\R_{0+}$-character, then the restriction of~$f$ to~$S$ gives a precentral distribution, so $f=f_\nu$ for the ergodic central measure~$\nu$ corresponding to this distribution.
\end{proof}

A character is called \emph{trivial} if it vanishes at all elements of the semigroup except the identity element. As one can easily see, the set of nontrivial $\R_{0+}$\nobreakdash-characters forms a fiber bundle over the set of normalized $\R_{0+}$-characters with fiber $(0,+\infty)$. Thus the problem of describing the absolute of commutative semigroups is essentially equivalent to the problem of describing the set of  $\R_{0+}$-charactes of commutative branching monoids. Correspondingly, assertions about the absolute of commutative semigroups can be translated into assertions about characters of branching semigroups. For instance, Theorem~\ref{th:abel-top-cancel} gives the following.

\begin{corollary}
\label{cor:abel-top-cancel-char}
For a commutative cancellative branching semigroup, the space of nontrivial $\R_{0+}$-characters endowed with the weak topology is homeomorphic to the direct product of a closed disk of certain dimension and an open interval.
\end{corollary}

In the case of commutative and nilpotent groups, there is a correspondence between the $\R_+$-characters of the group and the main part of the absolute. This correspondence is known, it has been studied from the point of view of eigenfunctions of Laplacians. We discuss it in more detail in a paper on the absolute of the Heisenberg group, which is currently in preparation.

\end{document}